\documentclass[a4paper,11pt,twoside,leqno]{article}
\usepackage{amsmath,amssymb,amsfonts,amsthm,graphicx,fancyhdr}
\usepackage[utf8]{inputenc}
\usepackage[T1]{fontenc}
\usepackage{rotating}
\usepackage[colorlinks=true]{hyperref}
\usepackage{thmtools}

\newtheorem{teo}{Theorem}[section]
\newtheorem{lem}[teo]{Lemma}
\newtheorem{prop}[teo]{Proposition}

\newtheorem{dfn}[teo]{Definition}
\newtheorem{ques}[teo]{Question}

\declaretheoremstyle[
  spaceabove=\topsep, spacebelow=\topsep,
  headfont=\bf,  
  notefont=\mdseries, notebraces={(}{)},
  bodyfont=\rmfamily, 
  postheadspace=1em,
  qed=$\Diamond$
]{drem}
\declaretheorem[style=drem, name=Remark, numberlike=teo]{rmk}

\newcommand{\eg}[0]{\emph{e.g.} }
\newcommand{\ie}[0]{\emph{i.e.} }

\newcommand{\srl}[1]{\overline{#1}}
\newcommand{\jo}[1]{\mathcal{#1}}

\DeclareFontFamily{T1}{mafra}{}
\DeclareFontShape{T1}{mafra}{m}{n}{<->s*[0.95]yswab}{} 
\DeclareFontShape{T1}{mafra}{m}{it}{<->s*[1.0]ygoth}{} 
\DeclareTextFontCommand{\textgoth}{\yfrak}

\DeclareSymbolFont{mafrak}{T1}{mafra}{m}{n}
\DeclareSymbolFontAlphabet{\mathfr}{mafrak}

\DeclareSymbolFont{mbbold}{U}{bbold}{m}{n}
\DeclareSymbolFontAlphabet{\mathbbold}{mbbold}

\newcommand{\mc}[1]{\mathcal{#1}}
\newcommand{\ms}[1]{\mathsf{#1}}
\newcommand{\mf}[1]{\mathfr{#1}}

\newcommand{\pgen}[1]{\langle #1 \rangle}

\newcommand{\imp}[0]{\Rightarrow}

\newcommand{\eps}[0]{\varepsilon}

\newcommand{\rr}[0]{\ensuremath{\mathbb{R}}}
\newcommand{\zz}[0]{\ensuremath{\mathbb{Z}}}

\newcommand{\del}[0]{\ensuremath{\partial}}

\newcommand{\img}[0]{\mathrm{Im}\,}

\newcommand{\Id}[0]{\mathrm{Id}}

\newcommand{\un}[0]{\mathbbold{1}}

\newcommand{\comp}[0]{\mathsf{c}}
\newcommand{\perpud}[0]{{\, \textrm{\mbox{\raisebox{.7ex}{\:\begin{rotate}{180}\makebox(0,0){$\perp$}\end{rotate}}}} \,}}

\newcommand{\eqtag}[0]{\addtocounter{teo}{1} \tag{\theteo}}

\begin{document}

\renewcommand{\thefootnote}{\fnsymbol{footnote}}

\renewcommand{\thefootnote}{\arabic{footnote}}

\newcommand{\ud}{\tfrac{1}{2}}
\newcommand{\ut}{\tfrac{1}{3}}
\newcommand{\uq}{\tfrac{1}{4}}

\newcommand{\lmd}{\lambda}
\newcommand{\Lmd}{\Lambda}
\newcommand{\Gm}{\Gamma}
\newcommand{\gm}{\gamma}
\newcommand{\GM}{\Gamma}

\newcommand{\bsl}\backslash
\newcommand{\acts}{\curvearrowright}
\newcommand{\donc}{\rightsquigarrow}

\newcommand{\smdd}[4]{\big( \begin{smallmatrix}#1 & #2 \\ #3 & #4\end{smallmatrix} \big)}

\newcommand{\sgn}{\textrm{sgn}\,}

\begin{center}
\Large Radial isoperimetry and absence of harmonic functions with $\ell^p$-gradient.
\vspace*{1cm}

\centerline{\large Antoine Gournay\footnote{The author gratefully acknowledges partial support from the ERC AdG grant 101097307.}}
\end{center}

\vspace*{1cm}

\centerline{\textsc{Abstract}}

\begin{center}
\parbox{10cm}{{ \small 
In this paper we show that groups for which the probability of return of a random walk is bounded below by $K_1e^{-K_2n^\gamma}$ have no non-constant harmonic functions with gradient in $\ell^p$.
The proof relies on results from $\ell^p$-cohomology, a form of radial isoperimetry, transport patterns and revisiting some results of F{\o}lner.
\hspace*{.1ex} 
}}
\end{center}

\section{Introduction}\label{s-intro}

\newcommand{\hdc}[0]{t\!\mc{H}\!\ms{D}^c}
\newcommand{\hdp}[0]{t\!\mc{H}\!\ms{D}^p}
\newcommand{\hdq}[0]{t\!\mc{H}\!\ms{D}^q}

The subject matter of this paper is to investigate which graphs and Cayley graph of groups possess harmonic functions whose gradient (the difference of the values of the function at the end of the edges) belongs to $\ell^p$ or $c_0$. The motivation comes mainly from groups: for example when $p=\infty$, one gets the class of Lipschitz harmonic functions which is of known importance, \eg see Shalom \& Tao \cite{ST}. 

Furthermore, if a Hilbertian representation of the group has non trivial reduced cohomology in degree 1, then there exist a non-constant harmonic function on the Cayley graph.
The gradient of this harmonic function is related to the mixing property of this representation. 
For example, for a strongly mixing representation this yields (in any Cayley graph) a harmonic functions with gradient in $c_0$. 
Hence if a group has no harmonic function with gradient in $c_0$ in some Cayley graph, then the reduced cohomology in degree 1 of any strongly mixing representation is trivial; see \cite[\S{}2]{GJ} or \cite[\S{}3]{Go-mixing} for details and references. 

Lastly, let us mention the reduced $\ell^p$-cohomology in degree 1 (of a group or graph), an useful quasi-isometry invariant. 
Under some assumptions on the isoperimetry, the non-vanishing of this cohomology is equivalent to the presence of harmonic function with gradient in $\ell^p$; see \cite{Go-frontier} for details and references. An underlying question due to Gromov \cite[\S{}8.$A_1$.$A_2$, p.226]{Gro} is whether any amenable group has harmonic function with gradient in $\ell^p$, the answer being posiive for $p=2$ by a result of Cheeger \& Gromov \cite{ChG}.

Using some analysis of the $p$-spectral gap of finite graphs as well as estimates on the separation profile from \cite{CG}, an appreciable family of groups may be covered.
\begin{teo}[see Theorem \ref{tteoprinc}]\label{tisop}
Let $G$ be a group and consider some Cayley graph of $G$. Assume there are constants $K_1,K_2 >0$ and $\gamma \in ]0,1[$ such that the probability of return of a random walk is bounded below by $K_1e^{-K_2n^\gamma}$.
Then there are no non-constant harmonic functions with gradient in $\ell^p$ on a Cayley graph of $G$.
\end{teo}
Since the probability of return only increases by taking quotient, note that the Theorem applies to any quotient of $G$ (see Remark \ref{remredu}).
This result covers any metabelian group as well as some non subexponentially elementarily amenable groups like the Basilica group (see Bartholdi \& Virág \cite{BV}).

The hypothesis on the isoperimetric ratio function is fairly weak (there are many groups which will satisfy them, see \eg \cite[\S{}4.2]{CG} or Pittet \& Saloff-Coste \cite[Theorem 7.2.1]{PSC-survey} and references therein).

The article is organised as follows.
\S{}\ref{sisolapent} develops tools which are necessary for Theorem \ref{tisop}:
\S{}\ref{sslaplp} bridges estimate on the Cheeger constant in finite graphs with the $\ell^p$-spectral gap 
and \S{}\ref{ssisoperel} focuses on the existence of useful finite sets inside groups, relying on the work \cite{CG}
These finite set as well as the bounds on the spectral gap will be used to construct transport plans. An interesting isoperimetric inequality \eqref{radiso} is introduced.
This hypothesis that has been dubbed ``radial isoperimetric inequality'' is a strengthening of a known inequality conjectured by Sikorav and proved by \.{Z}uk \cite{Zuk}.
\S{}\ref{sscontrex} gives a counterexample to show that \eqref{radiso} can only be expected to hold for optimal F{\o}lner sets. \S{}\ref{sslapl1} is a small excursion in the properties of the Laplacian when $p=1$.

\S{}\ref{strancoho} contains the bulk of the proofs.
In \S{}\ref{ssredcoh} results from \cite{Go} are used to reduce the problem.
\S{}\ref{sstranreliso} contains the proof of Theorem \ref{tisop} under the radial isoperimetric hypothesis; it relies on the well-chosen sets from \S{}\ref{ssisoperel} and the bounds of \S{}\ref{sslaplp} to do so.
\S{}\ref{ssfoelner} revisits work of F{\o}lner to show how to avoid this hypothesis.

As mentionned towards the end of \S{}\ref{ssisoperel}, the radial isoperimetric inequality seems to hold in many groups. However, since it may only apply to optimal sets (see \S{}\ref{sscontrex}), it seems difficult to verify it. This motivates the following question.
\begin{ques}\label{quesradiso}
For which [class of] groups does \eqref{radiso} hold? Are there groups where it fails?
\end{ques}
The question as to which groups have F{\o}lner sequences in which \eqref{radiso} holds should be easier (to answer in the positive).


%

\section{Isoperimetry and the Laplacian in $\ell^pV$}\label{sisolapent}

\newcommand{\IS}[0]{\mathrm{IS}}
\newcommand{\IIS}[0]{\mathrm{IIS}}

\subsection{Basic Notations and Definitions}\label{ssnablaetco}

The conventions are that a graph $\Gamma = (V,E)$ is defined by $V$, its set of vertices, and $E$, its set of edges.
Since we are mainly interested in Cayley graphs of finitely generated groups, all graphs will be assumed to be of bounded valency and the set of vertices $V$ will always be assumed to be countable.
The set of edges will be thought of as a subset of $V \times V$. The set of edges will be assumed symmetric (\ie~$(x,y) \in E \imp (y,x) \in E$).
Functions will take value in $\rr$
Functions on $E$ will always be anti-symmetric (\ie~$f(x,y)= -f(y,x)$).
This said $\ell^pV$ is the Banach space of functions on the vertices which are $p$-summable, while $\ell^pE$ will be the subspace of functions on the edges which are $p$-summable.

The {\bfseries gradient} $\nabla:\rr^V \to \rr^E$ is defined by $\nabla g(x,y) = g(y) - g(x)$.


Harmonic functions will be introduced by way of the divergence operator.
For two finitely supported function $f$ and $g$ on a countable set $Y$, define the pairing $\langle f \mid g \rangle_Y = \sum_{y \in Y} f(y)g(y)$.
(The subscript $Y$ will often be dropped.)
Though at first only defined for finitely supported functions, this extends to larger spaces, \eg $f \in \ell^pV$ and $g \in \ell^{p'}V$ (as usual, $p'$ will denote the H\"older conjugate exponent of $p$, \ie~$p' = p/(p-1)$).

This allows to define the adjoint of the gradient $\nabla$, denoted $\nabla^*$ and called {\bfseries divergence}, by $\langle f \mid \nabla g \rangle_E = \langle \nabla^* f \mid g \rangle_V$.
More precisely, for $f:E \to \rr$, one finds
\[
\nabla^*f(x) = \sum_{y \in N(x)} f(y,x) - \sum_{y \in N(x)} f(x,y).
\]
where $N(x)$ denotes the neighbours of $x$ (\ie the  vertices $y$ so that $(y,x)$ is an edge). In particular,
\[
 \nabla^* \nabla f(x) = 2 \sum_{y \in N(x)} \big( f(y) - f(x) \big).
\]
Thus, harmonic functions are exactly the functions for which the divergence of the gradient is trivial (the 0 function).
Equivalently, these are the functions which satisfy the mean value property.
The standing hypothesis that the graph is connected is important in order that the only function with trivial gradient are constant functions.

The [other] standing hypothesis that the graph has bounded valency is crucial in order for the gradient to be a bounded operator (from $\ell^pV \to \ell^pE$).
Note that the identity $\pgen{ \nabla^* f \mid g} = \pgen{f \mid \nabla g}$ holds if $f \in \ell^p E$ and $g \in \ell^{p'}V$ (where $p'$ is the H\"older conjugate of $p$ and $\ell^\infty$ can be replaced by $c_0$).

Lastly the {\bfseries random walk operator} $P$ as follows $P: \rr^V \to \rr^V$ is defined by
\[
(Pf)(x) = \tfrac{1}{|N(x)|}  \sum_{y \in N(x)} f(y).
\]
In regular graphs of degree $d$ it is then easy to check that $\nabla^* \nabla = d(\Id - P)$.

The short notation $\Delta = \nabla^* \nabla$ for the {\bfseries Laplacian} operator will be used most of the time.
However due to reasons of coherence with the litterature, in finite graphs (so essentially in the upcoming section), the definition $\Delta = \Id -P$ will be prefered. Note that in regular graphs, this just change the operator by a scalar.

\subsection{Inverting the Laplacian in $\ell^p$}\label{sslaplp}

In this subsection only $d$-regular finite graphs are considered, with just an important remark for the non-regular case at the end.
Here $\Delta$ is the Laplacian with spectrum (as an operator) in $[0,2]$, \ie $\Delta = I-P$.

\begin{dfn}
Assume the graph $G= (V,E)$ is finite.
The \textbf{$p$-spectral gap} of $\Delta$ is the largest constant $\lambda_p$ in 
\[
\sum_{x\in V} f(x) =0 \implies \| \Delta^{-1} f\|_{\ell^pV} \leq \lambda_p^{-1} \|f\|_{\ell^pV}
\]
The \textbf{$p$-conductance} constant is the largest constant $\kappa_p$ in
\[
\sum_{x \in V} f(x) =0 \implies \| \nabla f \|_{\ell^p E} \geq \kappa_p \|f \|_{\ell^pV}
\]
\end{dfn}
The aim of this subsection is to show the various inequalities between these constants.

If $p=2$, $\lambda_2$ would be the first non-zero eigenvalue of $\Delta$. 
By decomposing a function into its level set, one can prove that $\kappa_1$ is the usual isoperimetric (or conductance or Cheeger) constant. 
Also 
\[\eqtag \label{eqk2l2}
\kappa_2^2 = d\lambda_2
\]
and the classical result relating isoperimetry to the $2$-spectral gap is 
\[\eqtag \label{eqcheeger}
\frac{\kappa_1^2}{2 d^2} \leq \lambda_2 \leq \frac{2\kappa_1}{d}.
\]
See (among many possibilities) Mohar \cite{Mohar}.
\begin{lem}\label{tlmd2p-l}
Let $\tfrac{1}{p'} = 1- \tfrac{1}{p}$ and $\bar{p} = \max\{p',p\}$. Then $\lambda_p \geq \frac{2}{\bar{p}}\lambda_2.$
\end{lem}
\begin{proof}
The easy bound on $\lambda_p$ follows by observing that (for functions $f$ with $\sum_{x \in V} f(x)=0$) $\Delta^{-1} = \sum_{n \geq 0} P^n$. Then, one has (again restricting to the space of functions with zero mean)
\[
\|P\|_{\ell^2\to \ell^2} = 1- \lambda_2 <1
\text{ while }
\|P\|_{\ell^1\to \ell^1} \leq 1
\text{ and }
\|P\|_{\ell^\infty \to \ell^\infty} \leq 1
\]
So that, by Riesz-Thorin interpolation, $\|P\|_{\ell^p \to \ell^p} <1$ for any $p \in (1, \infty)$. This suffices to see that the above series converges. More precisely, this gives
\[
\frac{1}{\lambda_p} \leq \frac{1}{1-(1-\lambda_2)^{2/\bar{p}}} \leq \frac{\bar{p}}{2\lambda_2},
\]
where the last inequality follows by Taylor-Lagrange.
\end{proof}

\begin{lem}
$\kappa_p \geq \frac{d^{1/p}}{2}\lambda_p$
\end{lem}
\begin{proof}
The implication
``$
\sum_{x\in V} f(x) =0 \implies \| \Delta^{-1} f\|_{\ell^pV} \leq \lambda_p^{-1} \|f\|_{\ell^pV}
$''
is equivalent (letting $f = \Delta g$) to 
``$
\sum_{x \in V} g(x) =0 \implies \lambda_p \| g\|_{\ell^pV} \leq \|\Delta g\|_{\ell^pV}
$''
Since $\Delta = \tfrac{1}{d} \nabla^* \nabla $ and $\|\nabla^*\|_{\ell^pE \to \ell^pV} = \|\nabla\|_{\ell^{p'}V \to \ell^{p'}E} \leq 2d^{1/p'}$, one gets
\[
\sum_x g(x) =0 \implies \lambda_p \| g\|_{\ell^p} \leq \tfrac{2}{d^{1/p}} \|\nabla g\|_{\ell^p} \qedhere
\]
\end{proof}
The next inequality is probably one of the easiest.
\begin{lem}
$2^{p-1}\kappa_1 \geq \kappa_p^p$
\end{lem}
Combined with \eqref{eqk2l2} it also gives a proof of a part of \eqref{eqcheeger}.
\begin{proof}
Let $F \subset V$ with $|F| \leq |V|/2$ and take $f= |F^\comp| \un_F - |F| \un_{F^\comp}$ where $\un_A$ is the characteristic function of a set $A$. 
Note that $f$ has zero sum, $\nabla f = |V| \un_{\del F}$ and $\|f\|_{\ell^pV} = (|F^\comp|^p|F|+|F|^p|F^\comp|)^{1/p}$.
Applying this to the inequality for $\kappa_p$ yields
\[
|V| |\del F|^{1/p} \geq \kappa_p (|F^\comp|^p|F|+|F|^p|F^\comp|)^{1/p}.
\]
or $\frac{|\del F|}{|F|} \geq \kappa_p^p \frac{|F^\comp|^p+|F|^{p-1}|F^\comp|}{(|F|+|F^\comp|)^p}$. 
Since $|F| \leq |F|^\comp$, the right-hand side can bounded by finding the minimum of $x \mapsto \frac{x^{p-1}+1}{(x+1)^p}$ for $x \in [0,1]$.
One then get that $\frac{|\del F|}{|F|} \geq \kappa_p^p/ 2^{p-1}$.
By optimising on $F$, one gets $\kappa_1 \geq \kappa_p^p /2^{p-1}$.
\end{proof}
Lastly, the upcoming inequality goes back to Matou\v{s}ek \cite[Proposition 3]{Matou}. 
The upcoming lemma uses the same proof (with the obvious extension to $p \leq 2$ and checking the constants).
\begin{lem}
Assume $\sum_{x \in V} f(x) =0$. Then
$
\|\nabla f\|_p \geq (\tfrac{2}{d})^{1/p'} \min(\tfrac{1}{2},\tfrac{1}{p}) \kappa_1 \|f\|_p.
$
In particular, $\kappa_p \geq \big( \frac{2}{d} \big)^{(p-1)/p} \frac{1}{\max \{2,p \}} \kappa_1$. 
\end{lem}
\begin{proof}
By applying another variant of the inequality for $\kappa_1$ (see Matou\v{s}ek \cite[Lemma 2]{Matou} or Lov\'asz \cite[Ex.11.30 on p.83, p.144 and p.472]{Lov}) to $f^p$, one gets
\[
\kappa_1 \|f\|_p^p \leq \| \nabla (f^p)\|_1.
\]
Because $ |x^p-y^p| \leq \max(1,\tfrac{p}{2}) |x-y | |x^{p-1} + y^{p-1}| \leq \max(1,\tfrac{p}{2}) |x-y | (|x|^{p-1} + |y|^{p-1})$, one has
\[
\kappa_1 \|f\|_p^p \leq \max(1,\tfrac{p}{2})  \pgen{ \nabla f \mid \nabla^+(|f|^{p-1})},
\]
where $\nabla^+:  \ell^q V \to  \ell^q E$ is defined by $\nabla^+\phi(x,y) := \phi(x) + \phi(y)$.
By H\"older's inequality, this is 
$
\leq \max(1,\tfrac{p}{2}) \|\nabla f\|_p \| \nabla^+(|f|^{p-1})\|_{p'}.
$
But now
\[
(|x|^{p-1}+|y|^{p-1})^{p'} \leq 2^{p'-1}(|x|^{p'(p-1)}+|y|^{p'(p-1)}) = 2^{p'-1} (|x|^{p}+|y|^{p}),
\]
so
\[
\| \nabla^+(|f|^{p-1})\|_{p'} \leq 2^{1/p} d^{1/p'} \|f\|_p^{p/p'} = 2^{1/p} d^{1/p'} \|f\|_p^{p-1}.
\]
Then
\[
 \kappa_1 \|f\|_p^p \leq 2^{1/p} d^{1/p'} \max(1,\tfrac{p}{2})  \|\nabla f\|_p = (\tfrac{d}{2})^{1/p'} \max(2,p) \|\nabla f\|_p . \qedhere
\]
\end{proof}
As a summary:
\begin{teo}\label{tkplpsu-t}~\\[-.01ex]
\begin{tabular}{rr@{\,}c@{\,}lrr@{\,}c@{\,}lrr@{\,}c@{\,}c@{\,}c@{\,}l}
\emph{\textbf{1)}}&$2^{p-1}\kappa_1$			&$\geq$	&$\kappa_p^p$	&
\emph{\textbf{2)}}&$\kappa_2^2$ 			&$=$	&$d \lambda_2$	&
\emph{\textbf{3)}}&$\max\{2,p\} d^{(p-1)/p} \kappa_p$	&$\geq$	&$2^{(p-1)/p} \kappa_1$ \\
\emph{\textbf{4)}}&$\kappa_p$ 				&$\geq$	&$d^{1/p} \lambda_p$	&
\emph{\textbf{5)}}&$\bar{p} \lambda_p$ 			&$\geq$	&$2 \lambda_2$	&
\emph{\textbf{6)}}&$4d\kappa_1$ 			&$\geq$	&$2d^2\lambda_2$& $\geq$&$\kappa_1^2$
\end{tabular}
\end{teo}
In particular, the $p$-spectral gap is bounded (above and below) by functions of $\lambda_2$ and $\kappa_1$:
\[
\frac{\kappa_1^2}{d^2\bar{p}} 
\overset{\textbf{6}}{\leq} \frac{2 \lambda_2}{\bar{p}} 
\overset{\textbf{5}}{\leq} \lambda_p 
\overset{\textbf{4}}{\leq} \frac{ 2 \kappa_p}{d^{1/p}} 
\overset{\textbf{1}}{\leq} \frac{ 2^{(2p-1)/p} \kappa_1^{1/p}}{d^{1/p}} 
\overset{\textbf{6}}\leq 2^{(4p-1)/2p} \lambda_2^{1/2p}
\]

\begin{rmk}\label{rreggraph}
In the upcoming sections the lower bound $\frac{\kappa_1^2}{d^2\bar{p}} \leq \lambda_p $ will be used. However, it will be used on a finite graph which is induced from a regular infinite graph. Hence this finite graph is not regular, which was the standing assumption in the current section. To this end, let us recall that (in addition to this lower bound being well-known), the regularity hypothesis was not actually used for this lower bound. Indeed, the estimate {\bfseries 5} is Lemma \ref{tlmd2p-l}, which relies solely on an interpolation argument (which is independent of the regularity) and the estimate {\bfseries 2} is \eqref{eqcheeger} which is also not dependent on regularity.

It seems highly probable that all the other estimates also admit variations for non-regular graphs, but this is not the main aim of the current work.
\end{rmk}

\subsection{The Laplacian in $\ell^1$ }\label{sslapl1}

This subsection only deals with connected infinite graphs of bounded degree. 
Kesten \cite{Kes2} showed that a [finitely generated] group is non-amenable) if and only if the Laplacian $\Delta: \ell^2V \to \ell^2V$ is invertible (in any Cayley graph).
This also holds for graphs (see \eg Woess \cite[ 10.3 Theorem]{Woe}).
Using the same interpolation trick as in Lemma \ref{tlmd2p-l}, one can then show that the same is true for $\Delta: \ell^pV \to \ell^pV$ as long as $p \in ]1,\infty[$.

Let $c_0 V$ (the closure in $\ell^\infty$-norm of the finitely supported function),
Let us start by recalling the image and kernel of the Laplacian.
Sometimes $\img_{\mathcal{X}} \Delta$ (likewise for $\ker$) will be used to denote the image of $\Delta: \mathcal{X} \to \mathcal{X}$.
\begin{prop}\label{tproplap-p}
Assume $G$ is an infinite connected graph of bounded degree.
\begin{enumerate}\renewcommand{\labelenumi}{\bf \arabic{enumi}.} \renewcommand{\itemsep}{-1ex}
\item In $\ell^pV$ (for $1 \leq p<\infty$) and $c_0V$, $\ker \Delta =\{0\}$.
\item In $\ell^\infty$, $\ker \Delta \supset \{r \un_V \mid r \in \rr\}$ where $\un_\Gm$ is the constant function. 
\item $\srl{\img_{\ell^1V} \Delta} \subset \ell_0^1 V = \{ f \in \ell^1V \mid \sum_{v \in V} f(v) =0 \}$.
\item $\srl{\img_{\ell^1V} \Delta} = \ell_0^1 V$ if and only if there are no non-constant bounded harmonic functions.
\item In $\ell^1$ and $\ell^\infty$, $\img \Delta$ is weak$^*$ dense.
\item In $\ell^pV$ (for $1 < p<\infty$) and $c_0V$, $\img \Delta$ is dense.
\item Let $X = \ell^1V, \ell^\infty V$ or $c_0V$. There are sequences $f_n \in X$, so that $\frac{ \|\Delta f_n\|_{X} }{\|f_n\|_{X}} \to 0$.
\item In $\ell^1$, $\ell^\infty$ and $c_0$, $\Delta$ has no bounded inverse and the image is not closed.
\end{enumerate}
\end{prop}
\begin{proof}
The first point is a consequence of the maximum principle. Harmonic functions (\ie elements of $\ker \Delta$) in $c_0V$ tend to $0$ at infinity. By the maximum principle, they are $0$ everywhere.
The second is trivial: constant functions are in $\ell^\infty V$ so the Laplacian has a kernel.

A consequence of {\bf 2} and of $\Delta^* = \Delta$ is {\bf 3}. Indeed, if $X^*$ is the dual of $X$ and $A \subset X^*$, recall that $A^\perpud = \{ x \in X \mid \forall a \in A, \langle a \mid x \rangle =0 \}$ is the {\bfseries annihilator} of $A$. It follows from classical considerations that $\srl{ \img_{\ell^1} \Delta} = (\ker_{\ell^\infty} \Delta)^\perpud$ (see \cite[Lemma 2.8]{Go-cuts}). Since  $\ell_0^1V$ is the annihilator of the constant function, the conclusion follows.
The same considerations yield also {\bf 4}: the existence of other elements in $\ker_{\ell^\infty} \Delta$ makes the annihilator of $\Delta \ell^1V$ larger.

The fifth follows from $\srl{ \img_{\ell^1} \Delta}^* = (\ker_{c_0} \Delta)^\perp$, $\srl{ \img_{\ell^\infty} \Delta}^* = (\ker_{\ell^1} \Delta)^\perp$ and the first point. Likewise for {\bf 6}.

For {\bf 7} with $X = c_0V$, consider some root $\mf{o} \in V$ and the balls $B_n$ centred at $\mf{o}$.
Let $f_n =  \sum_{i=0}^n \tfrac{1}{n+1} \un_{B_i}$.
Then $f_n$ is finitely supported, $\|f_n\|_{\ell^\infty} = f_n(\mf{o}) = 1$ and $\|\nabla f_n\|_{\ell^\infty} \leq \tfrac{1}{n+1}$ (so, in particular $\Delta f_n \overset{\ell^\infty}\longrightarrow 0$). 

When $X = \ell^\infty V$, the same sequence works for the proof {\bf 7}. One do need to be careful with the kernel of $\Delta$. 
Namely, the norm of $g_n$ in the quotient space $\ell^\infty V/ \ker \Delta$ should not tend to $0$. 
Elements of $\ker \Delta$ are never $0$ at infinity.
Indeed, if the value of such an element is non-zero at $\mf{o}$, then it must take this value at infinity. 
Since the $g_n$ are $0$ at infinity (they are in $c_0V$), one gets that $\|g_n\|_{\ell^\infty V/\ker \Delta } \geq 1/2$.

The case $\ell^1$ in {\bf 7} is done here by a sequence which is reminiscent of Green's kernel.
Let $f_n = \frac{1}{n+1} \sum_{i=0}^n P^n \delta_x$ where $P$ is the random walk operator. 
Since $P^n \delta_x$ are the [positive] probability distributions of the random walk starting at $x$, $\|f_n\|_{\ell_1} =1$.
Now $\Delta f_n = (I-P)f_n = \frac{1}{n+1}(I-P^{n+1})\delta_x = \frac{1}{n+1} \big( \delta_x - P^{n+1}\delta_x \big)$. 
By linearity of the norm and the triangle inequality, $\|\Delta f_n\|_{\ell^1} \leq \frac{2}{n+1}$. This gives the claim.

The last point {\bf 8} is a direct consequence of {\bf 1} and {\bf 7}. 
\end{proof}
It is well established that $\img \Delta$ is closed in $\ell^p$ (for $1<p<\infty$) if and only if the graph is non-amenable (see for example Lohoue \cite{Lohoue}).

\subsection{Separation, relative isoperimetry and geometry of optimal sets}\label{ssisoperel}

\newcommand{\Gc}[0]{\widetilde{\jo{G}}}
\newcommand{\Go}[0]{\jo{G}_0}
\newcommand{\Gd}[0]{\jo{G}^{\scalebox{.4}{$\searrow$}}}

This section is a short review of following question: given an infinite graph with a rather large isoperimetry, does the same holds for the subgraphs induced on its subsets? As shown in \cite{CG}, given reasonable upper and lower bounds on the isoperimetric ratio function one can get lower bounds on the Cheeger constant on some finite subsets. 

The \textbf{isoperimetric function} is the function $\jo{F}: \zz_{>0} \to \zz_{\geq 0}$ defined to be the largest function so that for any set finite $F \subset V$, $\jo{F}(|F|) \leq |\del F|$, \ie $\jo{F}(x) = \inf \{ |\del F| \, \mid \, |F|= x\}$.
The \textbf{isoperimetric ratio function} is the function $\jo{G}: \zz_{> 0} \to \rr_{\geq 0}$ defined by $\jo{G}(x) := x^{-1}\jo{F}(x) = \inf \{ \frac{|\del F|}{|F|} \, \mid \, |F|= x\}$.
Be aware that some texts use ``isoperimetric function'' in place of ``isoperimetric ratio function''
The (decreasing) function $\Gd(x) := \inf \{ \frac{|\del F|}{|F|} \, \mid \, |F|\leq  x\}$ derived from $\jo{G}$ will also come in handy.

Looking at the decreasing function $\Gd$ instead of $\jo{G}$ is not much of thing.
Recall that $\jo{F}$ is subadditive: $\jo{F}(a+b) \leq \jo{F}(a) + \jo{F}(b)$. Hence, up to a minor change (taking the concave hull of a subadditive function changes it by at most a factor $2$), $\jo{F}$ is concave.
If $f$ is concave, then $x^{-1} f(x)$ is decreasing. 

A graph has a ... \\
\begin{tabular}{@{\textbullet\;}llll}
strong isoperimetric ratio& ($\IS_\omega$) & if $\exists K >0$ so that & $\jo{G}(x) \geq K $. \\
$d$-dimensional isoperimetric ratio&($\IS_d$) &if $\exists K >0$ so that &$\jo{G}(x) \geq \frac{K}{x^{1/d}}$. \\
$\nu$-intermediate isoperimetric ratio& ($\IIS_\nu$) &if $\exists K >0$ so that & $\jo{G}(x) \geq \frac{K}{(1+\ln x)^{1/\nu}}$. 
\end{tabular}

$\IS_\omega$ is equivalent to non-amenability. 
Typical graphs with $\IS_d$ are Cayley graph of $\zz^d$. 
For Cayley graphs of groups, only those of virtually nilpotent groups do not have $\IS_{d+1}$ for some $d$ (this $d$ is related to the growth of balls, but not to the dimension of the continuous counterpart of the group). 
There are no current example of a group which has none of the above three properties, \ie all known groups which are neither amenable nor virtually nilpotent satisfy some $\IIS_\nu$. Some groups (like iterated lamplighters) satisfy even stronger isoperimetric inequalities between $\IS_\omega$ and $\IIS_\nu$.

On a finite graph $(V,E)$, the functions $\jo{F}$ and $\jo{G}$ will only be considered to be defined on the range $0 \leq x \leq |V|/2$.
For example, the isoperimetric constant of a finite graph (also called Cheeger constant) is  $\kappa_1 := \min \{ \jo{G}(x) \mid 0 < x \leq |V|/2 \}$.

Note that, when speaking of isoperimetry, some authors consider rather: the (increasing) function $\jo{F}^{\scalebox{.4}{$\nearrow$}}(x) = \inf \{ |\del F| \, \mid \, |F| \geq x\}$, the (increasing) function $J^{\scalebox{.4}{$\nearrow$}}(x) = \Gd(x)^{-1}$ or $J = \jo{G}^{-1}$.

In \cite[\S{}3 and \S{}4]{CG} the authors develop a machinery to obtain lower bounds on the Cheeger constant of some sets. More precisely, if lower and upper bounds on $\jo{G}$ are available, then there are infinitely many finite sets $F$ for which $\kappa_1$ is bounded below. The methods cover quite a wide range of groups, so let us focus on the strictly necessary set-up and on a particular family.

The sets $F$ for which a lower bound on $\kappa_1(F)$ is given are [a subset of] {\bfseries optimal sets}. A set is said to be optimal (w.r.t. to isoperimetry) if $\jo{G}(|F|) = \frac{|\del F|}{|F|}$.
In other words, for any other set $F'$ with $|F'| \leq |F|$ one has $\frac{|\del F'|}{|F'|} \geq \frac{|\del F|}{|F|}$.
Although optimal sets have essentially not been explicitly given for most graphs (see \cite[Question 6.1]{CG}), there are many of them in Cayley graphs.

Indeed, if $n$ is an integer so that there is an optimal set with cardinality $n$, then the next such integer is at most $2n$ (since taking two disjoint copies of an optimal set yield a set with the same isoperimetric ratio).
As $n=1$ is an integer with an optimal set, one gets that there are many optimal sets.


The results of \cite[\S{}3 and \S{}4]{CG} will be used as follows: when the isoperimetric ratio function of a [Cayley graph of a] group satisfies $K_1 (\log n)^{-b} \leq \jo{G}(n) \leq K_2 (\log n)^{-a}$ (where $a < b \in \rr_{>0}$ and $K_1,K_2 \in \rr_{>0}$, then for infinitely many optimal sets $F$,
$\kappa_1(F) \geq K_3 \jo{G}(|F|) / \log(|F|)$ (for some $K_3 \in \rr_{>0}$). 

Such groups are nice, in the sense that, for some optimal sets, $\kappa_1(F)$ and $\jo{G}(|F|)$ are comparable in size.
There are however groups (like $C_2 \wr( C_2 \wr \zz )$) where $\kappa_1(F)$ decays much faster than $\jo{G}(|F|)$.

\newcommand{\inrad}[0]{\mathrm{inrad}}

Let us also discuss two further quantities which will be relevant in this context.
Given a subset $F$ of a graph $G$, define its {\bfseries inradius} to be $\inrad(F) = \max \{ r \mid  \exists x \in F$ s.t. $B_x(r) \subset F \}$ (here $B_x(r)$ is the ball of radius $r$ around $x$). The {\bfseries diameter} $\delta(F)$ of a conncected set $F$ is the maximal combinatorial distance of the graph induced on $F$.

For the following lemma the minimal and maximal volume growths ($f_v(n) := \inf_{x \in G} |B_n(x)|$ and  $f_V(n):=\sup_{x \in G} |B_n(x)|$ respectively) will come in handy. Note that both functions are increasing.
\begin{lem}
Assume $F$ is a connected subset of $G$. Then $\delta(F) \geq f_V^{-1}(|F|)$ and $\inrad(F) \leq f_v^{-1}(|F|)$.
\end{lem}
Using the notation $f(n) \preccurlyeq g(n)$ if there are constants $K, k>0$ such that $f(n) \leq K \cdot g(k \cdot n)$, note that
if $e^{n^\beta} \preccurlyeq  f_v(n) \preccurlyeq e^{n^\alpha}$ then $\delta(F) \succcurlyeq (\log |F|)^{1/\alpha}$ and $\inrad(F) \preccurlyeq (\log |F|)^{1/\beta}$.
\begin{proof}
Note that if the graph induced on the set $F$ is connected, $r = \inrad(F)$ is its inradius and $\delta$ is its diameter, then $|B_r(x)| \leq |F| \leq |B_\delta(x)|$ for some $x \in F$. Hence $f_v(r) \leq |F| \leq f_V(\delta)$ and the conclusion follows.
\end{proof}

Since any infinite connected graph contains a geodesic ray, there is no better generic bound than $\delta(F) \leq |F|$ for a connected set $F$. For optimal sets, one can do better:
\begin{lem}
Assume $G$ is the Cayley graph of a group and $K_1 (\log n)^{-b} \leq \jo{G}(n) \leq K_2 (\log n)^{-a}$ (where $a \leq b \in \rr_{>0}$ and $K_1,K_2 \in \rr_{>0}$. 
Then there are infinitely many optimal sets $F$ such that
$\delta(F) \leq K \frac{  (\log |F|)^2}{ \jo{G}(|F|) } \leq K' (\log |F|)^{2+b}$ where $K$ and $K'$ are constants depending only on the isoperimetric ratio function and the degree of $G$. 
\end{lem}
\begin{proof}
Using arguments from \cite[Proposition 5.5]{CG}, one has that for a connected set $F$ (whose diameter is $\geq 3$) the diameter of $F$ is bounded by $\delta \leq \frac{ 3 k \log |F|}{\kappa_1(F) \log 2}$ (where $k$ is the maximal degree of the graph induced on $F$).

Using \cite[\S{}3 and \S{}4]{CG} one gets when the isoperimetric ratio function of a [Cayley graph of a] group satisfies $K_1 (\log n)^{-b} \leq \jo{G}(n) \leq K_2 (\log n)^{-a}$ (where $a < b \in \rr_{>0}$ and $K_1,K_2 \in \rr_{>0}$, then for infinitely many optimal sets $F$, 
$\kappa_1(F) \geq \jo{G}(|F|) / \log(|F|)$. 

Putting these together yield the conclusion.
%
\end{proof}

Note that if the Cayley graph has a lower bound on the size of balls of the form $|B_r(x)| \geq K_3 e^{K_4 r^{\beta}}$, then $b$ (in the statement of the previous lemma) can be taken to be $\tfrac{1}{\beta}$.

Also if the separation profile is known, one gets a similar lower bounds on the sets which are optimal with respect to the separation profile. 
There is however no reason to believe that these sets are the optimal sets from the isoperimetric perspective.

There are also similar bounds for the inradius.
\begin{lem}
Assume $F$ is a finite optimal set, then $\inrad(F) \geq f_V^{-1}\big(   \tfrac{1}{\jo{G}(|F|)} \big)$. 
\end{lem}
\begin{proof}
Let $r$ be the inradius of $F$, then a ball of radius $r$ has at most  $f_V(r)$ elements.
Since $F \subset \cup_{x \ in \del F} B_r(x)$, one has that $|F| \leq f_V(r) |\del F|$.
The conclusion follows for optimal sets since $\frac{|F|}{|\del F|} = \jo{G}(|F|)^{-1}$. 
\end{proof}
In particular, when the growth is exponential (the ``worse case'' for a graph of bounded degree) $\inrad(F) \succcurlyeq -\log \jo{G}(|F|)$. 

Note that the lower bound $\inrad(F) \succcurlyeq -\log \jo{G}(|F|)$ is sharp for some graphs. An example is to consider the graph given by attaching the leaves of a binary tree to $\zz_{\geq 0}$. 
This graph has an isoperimetric ratio function which satisfies $\jo{G}(2^n-1) = 1/(2^n-1)$ and the corresponding optimal sets are easy to find. 
It can be turned into a regular graph (by taking three copies of it and gluing the leaves together).

The author expects a much better lower bound for vertex-transitive graphs. Indeed, for groups of polynomial growth one has $\jo{G}(n) \preccurlyeq n^{-1/d}$ and $f_V(n) = n^d$; so the above lemma yields only $n^{1/d^2}$ (whereas $n^{1/d}$ is expected).

Since a better bounds might be possible in groups, a property, from which such a bound would follow, will be introduced. Say a graph satisfies a {\bfseries radial isoperimetric inequality} if, there exists constants $K \geq 1$ and $k \geq 1$ so that, for any finite set $A$ which is optimal for isoperimetry,
\[\eqtag \label{radiso}
K |\del A| (1+ \inrad(A))^k \geq |A|
\]
Note this inequality is weaker than a strong isoperimetric inequality (since $k=0$ would suffice in that case).
Also graphs of polynomial growth satisfy this inequality trivially (with $K$ and $k$ such that $|B(n)| \leq K n^k$). 

For Cayley graphs, the inequality might be true with $k=1$ independently of the group.
Note that if one replaces the inradius by the diameter, then the inequality holds for $K=k=1$ (see \.{Z}uk \cite{Zuk}).

\begin{lem}\label{tleminrad}
Assume $G$ is a one-ended Cayley graph where \eqref{radiso} holds for some $k,K\geq 1$. Then for any optimal set $F$, $\inrad(F) \geq \tfrac{1}{K \jo{G}(|F|)^{1/k}}-1$.
\end{lem}
\begin{proof}
The conclusion follows for optimal sets since $\tfrac{|F|}{|\del F|} = \tfrac{1}{\jo{G}(|F|)}$
\end{proof}



Lastly, let us give a bound on the average distance to the boundary: given $F$ a finite connected set, let $\bar{r}(F) = \tfrac{1}{|F|} \sum_{x \in F} d(x,\del F)$.


\begin{rmk}\label{avradup}
It is not difficult to show that if $G$ satisfies $\IIS_\nu$ (with constant $k$), then $\bar{r}(F) \leq k (\ln|F|)^{1/\nu}$. Also a similar argument with the $d$-dimensional isoperimetric ratio function $\IS_d$ yields $\bar{r}(F) \leq k |F|^{1/d}$.

Since the inner radius is less than the average distance to the boundary, note that assuming \eqref{radiso} implies that these upper bounds are sharp.
\end{rmk}

For comparison, in the usual F{\o}lner sequences [which may not be optimal!] used for groups of intermediate growth, polycyclic groups (which are not nilpotent), wreath products $A \wr N$ (where $A$ is finite and $N$ has polynomial growth of degree $d$) and $\zz \wr \zz$, the corresponding bound on the inner radius $r(F)$ are respectively $(\ln |F|)^c$ with $c>1$, $\ln |F|$, $(\ln |F|)^c$ with $c = \tfrac{1}{d}$ and $\ln|F|/ \ln\ln |F|$. So the estimate \eqref{radiso} likely holds for such groups.

More generally, it seems likely that in elementarily amenable there might be F{\o}lner sequences for which this estimates hold.
However, as the next section will show, the radial isoperimetric inequality is false if one removes the hypothesis that it applies to non-optimal F{\o}lner sequences.
Hence the difficulty not only lies in arguing that there are F{\o}lner sequences which satisfy this isoperimetric inequatlity, but that these F{\o}lner sequences consists in optimal sets.

\subsection{Radial isoperimetry and non-optimal sets}\label{sscontrex}

The aim of this section is to give an example of a sequence of sets which violate the inequality \eqref{radiso} (sequence which violate the inequality necessarily make a F{\o}lner sequence). This means that the requirement that $A$ is optimal in \eqref{radiso} is necessary. As mentionned in the introduction, it might still be useful to know if there are F{\o}lner sequences for which \eqref{radiso} holds.

The group considered will be the lamplighter group $C_2^{(\zz)} \rtimes \zz$. Elements of this group are denoted $(f,z)$ where $f \in C_2^{(\zz)}$ is a finitely supported function from $\zz$ to $C_2$ and $z \in \zz$. The support of $f$ is often described as the lamps which are ``on'' while $z$ is the position of the lamplighter.

The Cayley graph considered here is the one with the ``switch or walk'' generating set. That is one can either change the state of the lamp at the current position of the lamplighter or move the lamplighter to another position.

The F{\o}lner sequence in this groups, even the optimal ones, are well understood, see \cite{Sta}.
Here are the important ingredients for this example.
Let $F_n = \{ (f,z) \mid z \in [1,n]$ and $\mathrm{supp} f \subset [1,n] \}$. Then $|F_n| = n 2^n$ and $|\del F_n| = 2 \cdot 2^n$.

Note that when $j$ divides $n$, there is a paving of translates of $F_j$ which cover $F_n$.
These translates are of the following form:
let $0 \leq i \leq n/j -1$, let $I_i = [ij+1,(i+1)j] \subset \zz$, then one translate is given by $i$ as well as a lamp state $f_0$ on $[1,n] \setminus I_i$. It is then of the form $(f+f_0,\ell)$ where $\ell \in I_i$ and $\mathrm{supp} f \subset I_i$.

Since $|F_n| = n 2^{n}$ there are $\tfrac{|F_n|}{|F_j|} = \tfrac{n}{j} 2^{n-j}$ translates required.
The diameter of $F_j$ is at most $4j$.
Also for any element of $F_n$, any ball of radius $4j$ contains some translate of $F_j$.

These are the basic remarks required to construct the (non-optimal) F{\o}lner sequence which is a counterexample to \eqref{radiso}.
In each in each translate of $F_j$ inside $F_n$ pick an element, and let $X$ be the set of the elements picked.
$F_n \setminus X$ is a good candidate to contradict the inequality, but in order to make the counterexample more convincing, a set whose complement is connected will be constructed. For this, it is needed to remove more points.

To do so let us consider an auxiliary graph. The vertices are the translates of $F_j$ as well as $F_n^\comp$. To elements are related if the are related by an edge (meaning there is an edge with one end in each set) in the Cayley graph.
This graph has $\tfrac{|F_n|}{|F_j|} +1= \tfrac{n}{j} 2^{n-j} +1$ vertices.
Take a covering tree $T$; it has $\tfrac{n}{j} 2^{n-j}$ edges.
Now realise each edge of $T$ as a path in the Cayley graph between the vertices of $X$.
Each edge is realised by a path of length at most $8j+1$, since the diameter of each $F_j$ is at most $4j$.
Let $F_{n;j}$ be the set $F_n$ where all these paths are removed.
The set of removed vertice has cardinality $ \leq (8j+1) \cdot \tfrac{n}{j} 2^{n-j} \leq 9  n 2^{n-j}$

The rest is just a computation. $|F_{n;j}|  \geq |F_n| - 9n 2^{n-j} = n 2^n (1- 9 \cdot 2^{-j})$.
On the other hand every removed vertex contributes at most $d$ further boundary edge, where $d$ is the degree of the graph (here $d=3$).
Hence $|\del F_{n;j}| \leq |\del F_n| + 3 \cdot 9n 2^{n-j} = 2^n (2- 27\cdot 2^{-j})$.
Lastly, the inradius of $F_{n;j}$ is at most $4j$ since any ball of this radius will contain a translate of $F_j$ and hence a removed vertex. Putting this together one gets that, for any constants $K$ and $k$,
\[
\frac{K |\del F_{n;j}| (1+ \inrad(F_{n;j}))^k}{|F_{n;j}|}
\geq
\frac {K 2^n (2- 27\cdot 2^{-j}) (5j)^k }{   n 2^n (1- 9 \cdot 2^{-j})}
\geq \frac{(5j)^k}{n} \frac{ K (2- 27\cdot 2^{-j})}{  (1- 9 \cdot 2^{-j})}
\]
As soon as $j \geq 5$, this is $\geq \frac{ (5j)^k}{n} K$.
Now it suffices to pick $j$ to grow logarithmically as $n$ grows, to see that this ratio is never bounded below whatever the choice of $K$ and $k$.

Hence taking $j = 2^i$ and $n = 2^j$ gives a F{\o}lner sequence which violates \eqref{radiso} for any $K$ and $k$.
This sequence is very obviously not optimal, so that question \ref{quesradiso} remains open.

\section{Transport and reduced $\ell^p$-cohomology in degree $1$}\label{strancoho}

\subsection{Reduced $\ell^p$-cohomology and harmonic functions}\label{ssredcoh}

Even if one is only interested in harmonic functions with gradient in $\ell^p$, it is worthwhile to consider reduced $\ell^p$-cohomology in degree one.
If a group has polynomial growth, it is not difficult to check that there are no harmonic functions with gradient in $\ell^p$ (in fact in $c_0$; see \cite[Proposition 1.5]{GJ})
Likewise its reduced $\ell^p$-cohomology is trivial (see \cite{Go-frontier} and references therein for degree 1 or \cite{Kappos} for all degrees).

This said here are some further results from \cite{Go} (see also \cite{Go-frontier}).
When the group does not have polynomial growth, then there is an equivalence between having trivial reduced $\ell^p$-cohomology in degree one for any $p \in [1,\infty[$ and the absence of harmonic with gradient in $\ell^p$ for any $p \in [1,\infty[$.
Furthermore, for any group which has a trivial Poisson boundary (in particular, any group which is not of exponential growth), the reduced $\ell^p$-cohomology in degree one is trivial.

Hence there only remain the case of groups of exponential growth (and many of those are already covered, see \cite[\S{}5.1]{Go-frontier} for an overview).
\begin{rmk}\label{risopretprob}
This reduction allows an important alleviation of the hypothesis in our main theorem.
Indeed, if a group has exponential growth, then the isoperimetry ratio function is already bounded above by $\frac{n}{\log n}$.

Having that the decay of the return probability of a random walk is bounded below by $K_1 \mathrm{exp}(-K_2 n^\gamma)$ (for some constants $K_1,K_2 >0$ and some $\gamma >0$) yields a lower bound on the isoperimetry in $\frac{n}{(\log n)^{(1-\gamma)/2\gamma}}$. (In groups of exponential growth $\gamma \leq 1/3$). See the survey by Pittet \& Saloff-Coste \cite{PSC-survey} or the work of Coulhon \& Saloff-Coste \cite{CSC} for details.

Hence the lower bound on the return probability implies that the isoperimetric ratio function is bounded above and below by funtions of the type $K n / (\log n)^k$.
\end{rmk}

The strategy in \cite{Go} relies on the idea of transport patterns, and since the current paper also relies on this concept, a small review will be made.
Note that for any $h \in \ell^1E$, $\nabla^* h =: \pi$ can be decomposed into two positive functions $\pi_\pm \in \ell^1V$ so that $\pi = \pi_+ - \pi_-$ and $\|\pi_+\|_{\ell^1} = \|\pi_-\|_{\ell^1}$.
This follows from the fact that the image of $\nabla^*$ is contained in the functions summing to 0.
Conversely, given two measures $\xi$ and $\mu$, a {\bfseries transport pattern} from $\xi$ to $\mu$ is a finitely supported (alternating) function on the edges $\tau: E \to \rr$ such that $\nabla^*\tau = \mu- \xi$.

The following lemma (whose proof is contained in the statement, being simply the definition of $\nabla^*$), is the reason why transport patterns are so useful in our context.
\begin{lem}\label{testimtp-l}
Let $f$ be a function such that $\nabla f \in \ell^pE$ and let $\tau$ be a transport pattern from $\xi$ to $\mu$. Then
\[
 \langle f \mid \mu \rangle -  \langle f \mid \xi \rangle
= \langle f \mid \mu - \xi \rangle
= \langle f \mid \nabla^*\tau \rangle
= \langle \nabla f \mid \tau \rangle
\leq \|\nabla f\|_{\ell^pE} \| \tau \|_{\ell^{p'}E}
\]
\end{lem}
This lemma will be used along the following lines. When $f$ is harmonic, its value at a vertex $x$ (resp. $y$) are the same as its values on some average $\mu$ (resp. $\xi$) around $x$ (resp. $y$): $f(x) = \langle f \mid \mu \rangle$.
To show $f$ is constant, we will show that for any neighbours $x$ and $y$, there is a sequence of transport $\tau_n$ plans between some averages $\mu$ and $\xi$, so that the norm of $\tau_n$ tends to 0.
This will show that $f(x) - f(y) =0 $ and imply that $f$ is constant.

Here is a slightly more detailed construction of the transport patterns.
First replace the Dirac masses by the a distribution of the random walk (thus reducing the norm).
Inverting the Laplacian (thus increasing the norm) will produce a transport pattern: indeed the divergence of $\nabla \Delta^{-1} (\xi - \mu)$ is just $(\xi -\mu)$.
The inversion of the Laplacian will require estimates on $\kappa_1$, which is why it will be required that the F{\o}lner sets are optimal;
see \S{}\ref{ssisoperel} or \cite{CG} for details. The decrease and increase in norm compensate correctly, to give the desired application of Lemma \ref{testimtp-l}

\subsection{Transport using radial isoperimetry}
\label{sstranreliso}

This section is meant as a warm-up to the proof of the main result, Theorem \ref{tteoprinc}, as well as a motivation to the introduction of the radial isoperimetric inequality \eqref{radiso}. In the next section, the dependance on this (hypothetical) radial isoperimetric inequality will be removed, by showing that this inequality in some sense holds for a random walk, although it might not hold geometrically.


In order to keep the constants legible, only the case $p \in [2,\infty[$ will be taken into account; this is however the only case of intereset, see \S{}\ref{ssredcoh}.

\begin{teo}\label{ttranindiso-t}
Let $G$ be the Cayley graph of a group and assume that for some constants $a,b, K_1$ and $K_2\in ]0,\infty[$, one has $\dfrac{K_1}{(\ln x)^a} \leq \jo{G}(x) \leq \dfrac{K_2}{(\ln x)^b}$.
Assume further that a radial isoperimetric inequality \eqref{radiso} holds.
Then there are no non-constant harmonic functions with gradient in $\ell^p$ for $p\in ]1, \infty[$.
\end{teo}
\begin{proof}
Assume $f$ is a harmonic function and consider the difference of two value at neighbouring vertices $c= f(v)-f(w) = \langle f \mid \delta_v-\delta_w \rangle$.
Without loss of generality, and to alleviate the notations, it will be assumed that $c \geq 0$ (otherwise switch the vertices $v$ and $w$.
Since $f$ is harmonic, $\langle f \mid \delta_v \rangle = \langle f \mid P^k\delta_v \rangle$ for any integer $k \geq 0$.

Consider an optimal sequence of F{\o}lner sets $F_n$ and let $r_n = \inrad(F_n)-1$. Then up to translating the sets $F_n$, one may assume that balls of radius $r_n$ around $x$ or $y$ lie inside the set $F_n$.

By Lemma \ref{testimtp-l}, one then gets $c = \langle f \mid P^{r(F_n)} (\delta_v-\delta_w) \rangle = \langle \nabla f | \tau_n \rangle \leq \|\nabla f\|_{\ell^q} \|\tau_n\|_{\ell^p}$.
Since the gradient of $f$ is bounded (in $\ell^q$-norm) this last term would also tend to 0, showing that $c=0$. Since $v$ and $w$ are arbitrary neighbours, this means that $f$ is constant.

Let $\Delta_n$ be the Laplacian restricted to graph induced on $F_n$, $\nabla_n$ be the gradient restricted to that same finite graph and $g_n = P^{r(F_n)} \delta_w - P^{r(F_n)} \delta_v$.
Note that $g_n$ has zero sum, hence lies in the image of $\Delta_n$.
Let $h_n = \Delta_n^{-1} g$.
Let $\tau_n$ be the function (on the edges) which is identically equal to $\nabla_n h_n$ inside $F_n$.
This is a transport pattern since $\nabla^* \tau_n$ only is non-zero on vertices in $F_n$.
Furthermore, on such vertices, it is equal to $\nabla^* \nabla_n h_n = \Delta_n h_n = g_n$.

It remains to check that the norm of $\tau_n$ tends to $0$.
For this note that $\|g_n\|_{\ell^p} \leq \|P^{r(F_n)} \delta_v\|_{\ell^p} + \|P^{r(F_n)} \delta_w\|_{\ell^p}$.
Furthermore,
\[
\|P^{r(F_n)} \delta_w\|_{\ell^p}^p \leq \|P^{r(F_n)} \delta_w\|_{\ell^\infty}^{p-1} \|P^{r(F_n)} \delta_w\|_{\ell^1} = \|P^{r(F_n)} \delta_w\|_{\ell^\infty}^{p-1}.
\]
Hence $\|g_n\|_{\ell^p} \leq  \|P^{r(F_n)} \delta_w\|_{\ell^\infty}^{1/q}$ (recall $q$ is the H\"older conjugate of $p$).
But $\|P^{r(F_n)} \delta_w\|_{\ell^\infty}$ is the return probability after time $r(F_n)$. So this part of the estimate decreases the norm of the transport pattern

Next
\[
\begin{array}{rll}
\|\tau_n\|_{\ell^pE}
&= \|\nabla h_n\|_{\ell^pE}
&\leq d \|h_n\|_{\ell^pV} \\
&\leq d \|\Delta_n^{-1}\|_{\ell^p \to \ell^p} \|g_n\|_{\ell^pV}
&\leq d \lambda_p(F_n)^{-1} \|g_n\|_{\ell^pV}\\
&\leq 2 q d^3  \kappa_1(F_n)^{-2} \|P^{r(F_n)} \delta_w\|_{\ell^\infty}^{1/q}
\end{array}
\]
where the last inequality follows from Theorem \ref{tkplpsu-t} and the previous estimate on the norm of $g_n$.

By Lemma \ref{tleminrad}, the radial isoperimetric inequality
\eqref{radiso} implies that
there are $c$ and $K_3 \in ]0,\infty[$ so that, for any optimal set $F$, the inner radius $r(F)$ is bounded below: $r(F_n) \geq K_3 (\ln |F_n|)^c$

Heat kernel estimates for such isoperimetric profiles also imply that the probability of return at time $k$ is at most $K_4 \mathrm{exp}(-K_5 k^\gamma)$ (see among many possibilities Pittet \& Saloff-Coste \cite{PSC-survey} or Coulhon \& Saloff-Coste \cite{CSC}). This yields $\|P^{r(F_n)} \delta_w\|_{\ell^\infty}^{1/q} \leq K_6 \mathrm{exp}(-K_7 (\ln |F_n|)^{c\gamma})$

As a consequence of \cite{CG} there are infinitely many optimal sets such that $\kappa_1(F) \geq K_8 / (\ln |F_n|)^t$ for some other constants $K_4$ and $t$. Hence
 $\kappa_1(F_n)^{-2} \leq K_9 (\ln |F_n|)^{2t}$.

Putting these estimates together and putting $x = (\ln |F_n|)^{c\gamma}$, the bound reads
\[
 \|\tau_n\|_{\ell^pE} \leq  K_{10} (\ln |F_n|)^{2t} \mathrm{exp}(-K_7 (\ln |F_n|)^{c\gamma}) \leq K_{10} x^{2t/ c \gamma} \cdot \mathrm{exp}(-K_7 x).
\]
Since $x \to \infty$ as $n \to \infty$, this tends to 0.
\end{proof}

\begin{rmk}
The proof work for a very slightly larger class of groups, since it is only required that $\kappa_1(F_n)^{-2} \rho_{r(F_n)}^{1/q}$ tend to 0 (where $\rho_k$ is the probability of return at time $k$).
\end{rmk}

\subsection{Transport without radial isoperimetry}\label{ssfoelner}

The following lemma, inspired by the work of F{\o}lner \cite{Fol55} (see also Cannon, Floyd \& Parry \cite{CFP} for a more recent account), will enable us to get rid of the hypothesis on radial isoperimetry.
Note that the upcoming lemma holds for generic graph, the hypothesis that all vertices have the same degree can probably be dropped.
\begin{lem}\label{tlemfolner}
Assume the graph is regular of degree $d$.
Let $\eps >0$ and let $F$ be a finite set (of vertices). Let $\un_F$ be the indicator function of $F$ (which takes value 1 on $F$ and 0 on its complement). Let $r$ be such that $\|P^r \un_F\|_{\ell^\infty} \leq 1-\eps$. Then $r \geq \frac{d}{2} \cdot \eps \cdot \frac{|F|}{|\del F|}$
\end{lem}
\begin{proof}
Let $f = \displaystyle \sum_{i=0}^{r-1} P^i \un_F$.
Then $\nabla^* \nabla f = d(1-P)f = d( \un_F - P^r \un_F)$.
By hypothesis, $\nabla^* \nabla f \geq d \eps$.
Hence
\[
 d\eps|F| \leq \langle \un_F \mid \nabla^*\nabla f \rangle = \langle \nabla \un_F \mid \nabla f \rangle \leq |\del F| \cdot \| \nabla f\|_{\ell^\infty( \del F)}
\]
However, if $x \in F$ and $y \notin F$, then
\[
\begin{array}{rl}
\nabla f(x,y)
= f(y)-f(x)
&= \langle f \mid \delta_{y} - \delta_x \rangle \\
&= \displaystyle \sum_{i=0}^{r-1} \langle P^i \un_F \mid \delta_{y} - \delta_x \rangle\\
&= \displaystyle \sum_{i=0}^{r-1} \langle \un_F \mid P^i(\delta_{y} - \delta_x) \rangle\\
&= \displaystyle  \langle \un_F \mid \sum_{i=0}^{r-1} P^i(\delta_{y} - \delta_x) \rangle\\
& \leq \|\un_F\|_{\ell^\infty} \|\displaystyle \sum_{i=0}^{r-1} P^i(\delta_{y} - \delta_x)\|_{\ell^1}\\
& \leq 2r
\end{array}
\]
where the last inequality follows by the triangle inequality (as $\displaystyle \sum_{i=0}^{r-1} P^i(\delta_{y} - \delta_x)$ is a sum of $2r$ probability measures). Hence $d \eps |F| \leq |\del F| 2 r$ which is the bound claimed.
\end{proof}

\begin{rmk}
Note that $r$ is necessarily larger than the inradius of $F$ as soon as $\eps >0$. Furthermore, as $\eps$ tends to 1, $r$ tends to $\infty$. Hence for fixed set $F$, consider the map from $\eps$ to $r_F(\eps)$ the smallest integer which satisfies the hypothesis of Lemma \ref{tlemfolner}

For $\eps>0$, consider also the map $\eps \mapsto 1/\sqrt{\eps}$. This map takes arbitrarily large value for $\eps \to 0$ and value 1 for $\eps = 1$. Consequently the graphs of $\eps \to r_F(\eps)$ and $\eps \to 1/\sqrt{\eps}$ cross each other.

Henceforth, $\eps_0(F)$ will be such that for $\eps < \eps_0$, $r_F(\eps) < 1/\sqrt{\eps}$ and for $\eps > \eps_0$, $r_F(\eps) < 1/\sqrt{\eps}$, and $r_0(F) := r_F(\eps_0)$.
\end{rmk}

\begin{teo}\label{tteoprinc}
Let $G$ be a group of exponential growth.
Assume that the probability of return of the simple random walk at time $k$ is bounded below by $K \mathrm{exp}(K'k^\gamma)$.
Then there are no non-constant harmonic functions with $\ell^p$-gradient in the Cayley graph.
\end{teo}
\begin{proof}
First by Remark \ref{risopretprob} the hypothesis imply that the isoperimetric ratio is bounded between two power of logarithms as in Theorem \ref{ttranindiso-t}.

The strategy will now be the same as in Theorem \ref{ttranindiso-t}, except that the hypothesis of the inradius can be discarded since now $r_0(F)$ will be used instead.
First note that, for some $y \in F$,
$1-\eps_0 \geq P^{r_0}\un_F(y) = \langle P^{r_0} \un_F \mid \delta_y \rangle = \langle \un_F \mid P^{r_0} \delta_y \rangle$.
This implies that after making $r_0$ steps from $y$, at most $1/r_0(F)^2$ of the mass left the set $F$. By minimality of $\eps \mapsto r_F(\eps)$, making one step less means less mass left $F$. Hence by taking $r'_0 = r_0(F)-1$, by making $r'_0$ steps of the random walk from $y$ or a neighbour $x$, the mass leaving $F$ is at most $1/r_0(F)^2$.

There are now two parts to consider when making a transport pattern for $P^{r_0(F)}(\delta_y-\delta_x)$.
The mass which has left $F$ can be transported by some paths of length at most $2r_0(F)+1$. This transport plan for the mass which left $F$ has $\ell^1$-norm $\frac{2r_0(F)+1}{r_0(F)^2}$ (and $\ell^\infty$-norm $\tfrac{1}{r_0(F)^2}$), hence it tends to 0 in $\ell^p$ (for any $p$). Note that is could happen that the mass which left $F$ is not balanced (some of ot has $-$ sign). But then one can compensate by taking also some mass inside $F$.

The mass inside the set $F$ can be transported as in Theorem \ref{ttranindiso-t}, by Lemma \ref{tlemfolner} and using $\eps_0(F)= 1/r_0(F)^2$, $r_0(F)^3 \geq \frac{d}{2} \cdot \frac{|F|}{|\del F|}$ (which replaces the bound from radial isoperimetry).
\end{proof}


\begin{rmk}\label{remredu}
Assume a group $G$ satisfies the estimates $e^{-n^\alpha} \leq \rho_n \leq e^{-n^\beta}$. Then by Theorem \ref{tteoprinc} and Remark \ref{risopretprob},
it has no non-constant harmonic function with gradient in $\ell^p$ for any $p \in ]1,\infty[$ (\ie the reduced $\ell^p$-cohomology is trivial).
Let $Q$ be a quotient of $G$. If $Q$ has subexponential growth, then there are no such harmonic function by \cite{Go} (see also \cite{Go-frontier}). If $Q$ has exponential growth, then note that there the upper bound on $\rho_n$ holds (with $\beta = 1/3$) while taking quotients only increase the return probability, hence the lower bound also holds with the same $\alpha$.
Hence the conclusion of Theorem \ref{tteoprinc} extends to any quotient of such groups.
\end{rmk}

There are also groups where it seems very unlikely that Theorem \ref{ttranindiso-t} applies.
A possible candidate could be as simple as $A \wr (A \wr \zz)$ where $A$ is some finite group.
There are two reasons for this.
First, \cite{CG} show that (even for optimal sets) $\kappa_1(F) \leq (\ln|F|)^c$ (which is much worse than the isoperimetric ratio function).
Second, the inner radius for the usual F{\o}lner sets is probably fairly small, \eg one would expect something like $r(F) \approx \ln \ln |F|$.

This indicates that the methods probably breaks for groups whose isoperimetry ratio function decays more slowly than a power of $\ln$, or equivalently for groups whose probability of return is closer to linear; \eg of the type $K_1 \mathrm{exp}\big(-K_2 n/(\ln n)^a\big)$.
Note however that for the group $A \wr (A \wr \zz)$ it is quite simple to show that there are no non-constant harmonic functions with gradient in $\ell^p$, see \cite{Go-cras}.

%


\begin{thebibliography}{10}


\bibitem{BV}
L.~Bartholdi and B~.Virág ,
\newblock Amenability via random walks,
\newblock \emph{Duke Math. J.} \textbf{130}(1):39--56, 2005.






\bibitem{CFP}
J. W. Cannon, W. J. Floyd, and W. R. Parry,
\newblock Amenability, Følner Ratios, and Cooling Functions,
\newblock arXiv:1201.0132

\bibitem{ChG}
J.~Cheeger and M.~Gromov,
\newblock {$L\sb 2$}-cohomology and group cohomology,
\newblock \emph{Topology}, \textbf{25}:189--215, 1986.



\bibitem{CSC}
T.~Coulhon and L.~Saloff-Coste
\newblock Isop\'erim\'etrie pour les groupes et les vari\'et\'es,
\newblock Revista Matem\'atica Iberoamericana \textbf{9}(2):293--314, 1993.

\bibitem{CG}
C.~Le Coz and A.~Gournay,
\newblock Separation profile, isoperimetry, growth and compression,
\newblock to appear \emph{Ann. Inst. Fourier}, arXiv:1910.11733.


%


\bibitem{Fol55}
E.~F{\o}lner,
\newblock On groups with full Banach mean value,
\newblock \emph{Math. Scand.} \textbf{3}:243--254, 1955.

 
\bibitem{Go}
A.~Gournay,
\newblock Boundary values of random walks and $\ell^p$-cohomology in degree one,
\newblock \emph{Groups Geom. Dyn.} \textbf{9}(4):1153--1184, 2015. 

\bibitem{Go-cras}
A.~Gournay,
\newblock Harmonic functions with finite $p$-energy on lamplighter graphs are constant.,
\newblock \emph{C. R. Acad. Sci. Paris} \textbf{354}:762--765, 2016.

\bibitem{Go-mixing}
A.~Gournay,
\newblock Mixing, malnormal subgroups and cohomology in degree one, 
\newblock \emph{Groups Geom. Dyn.} \textbf{12}(4):1371--1416, 2018.


\bibitem{Go-frontier}
A.~Gournay, 
\newblock Linear and Nonlinear Harmonic Boundaries of Graphs; An Approach with $\ell^p$-Cohomology in Degree One,
\newblock In: \emph{Frontiers in Analysis and Probability}, Anantharaman N., Nikeghbali A., Rassias M.T. (eds), 47--76:2020. 

\bibitem{Go-cuts}
A.~Gournay,
\newblock Cuts, flows and gradient conditions on harmonic functions,
\newblock arXiv:2206.13275



\bibitem{GJ}
A.~Gournay and P.-N.~Jolissaint,
\newblock Functions conditionally of negative type on groups acting on regular trees,
\newblock \emph{J. London Math. Soc.} \textbf{93}(3):619--642, 2016. 


\bibitem{Gro}
M.~Gromov,
\newblock Asymptotic invariants of groups, \textit{in} \emph{Geometric group theory ({V}ol. 2)},
\newblock London Mathematical Society Lecture Note Series, {V}ol. \textbf{182}, Cambridge University Press, 1993, viii+295.



\bibitem{Kappos}
E.~Kappos,
\newblock $\ell^p$-cohomology for groups of type $\mathsf{FP}_n$.
\newblock arXiv:math/0511002 , 2006.




\bibitem{Kes2}
H.~Kesten,
\newblock Full Banach mean values on countable groups, 
\newblock \emph{Math. Scand.}, \textbf{7}:146--156, 1959.


\bibitem{Lohoue}
N.~Lohou\'e, 
\newblock Remarques sur un th\'eor\`eme de Strichartz. 
\newblock \emph{C. R. Acad. Sci., Ser.I}, \textbf{311}:507--510, 1990.


\bibitem{Lov}
L.~Lov\'asz, 
\newblock \emph{Combinatorial Problems and Exercises} (2nd ed.),
\newblock North-Holland Publishing Co., 1993.


\bibitem{Matou}
J.~Matou\v{s}ek,
\newblock On embedding expanders into $\ell_p$ spaces,
\newblock \emph{Israel J. Math.} \textbf{102}:189--197, 1997.


\bibitem{Mohar}
B.~Mohar, 
\newblock Isoperimetric inequalities, growth, and the spectrum of graphs,
\newblock \emph{Linear Algebra Appl.} \textbf{103}:119--131, 1988.



 
 
\bibitem{PSC-survey}
C.~Pittet and L.~Saloff-Coste,
\newblock A survey on the relationships between volume growth, isoperimetry, and the behavior of simple random walk on Cayley graphs, with examples,
\newblock \emph{unpublished manuscript}, 2001.




\bibitem{ST}
Y.~Shalom and T.~Tao,
\newblock A finitary version of Gromov's polynomial growth theorem,
\newblock \emph{Geom. Funct. Anal.} \textbf{20}(6):1502--1547, 2010. 

\bibitem{Sta}
B.~Stankov,
\newblock Exact descriptions of F{\o}lner functions and sets on wreath products and Baumslag-Solitar groups,
\newblock arxiv.org/abs/2111.09158

\bibitem{Woe}
W.~Woess,
\newblock \emph{Random Walks on Infinite Graphs and Groups}, Cambridge tracts in mathematics, \textbf{138}. 
\newblock Cambridge University Press, 2000. 

\bibitem{Zuk}
A.~\.{Z}uk,
\newblock On an isoperimetric inequality for infinite finitely generated groups,
\newblock \emph{Topology} \textbf{39}:947--956, 2000.



\end{thebibliography}
\end{document}